\documentclass[12pt]{amsart}

\usepackage{amssymb, amsmath, amscd, eucal, rotating}
\usepackage{graphicx}
\usepackage{times,epsf,epsfig,color}
\usepackage{hyperref}
\usepackage{tikz-cd}
\usepackage{fullpage,enumerate}
\usepackage{amsmath,  amssymb, amscd}
\usepackage{mathtools}

\numberwithin{equation}{section}

\theoremstyle{plain}

\newtheorem{proposition}{Proposition}[section]
\newtheorem{theorem}[proposition]{Theorem}		
\newtheorem*{theorem*}{Theorem}		
\newtheorem{corollary}[proposition]{Corollary}
\newtheorem{lemma}[proposition]{Lemma}

\theoremstyle{definition}

\newtheorem{definition}[proposition]{Definition}
\newtheorem{remark}[proposition]{Remark}

\theoremstyle{remark}

\newcommand{\C}{\mathbb{C}}

\newcommand{\ZZ}{\mathbb Z}

\DeclareMathOperator{\id}{id}
\DeclareMathOperator{\Pic}{Pic}
\DeclareMathOperator{\Div}{Div}
\DeclareMathOperator{\rank}{rank}
\DeclareMathOperator{\vol}{vol}

\begin{document}

\title[Fractional quantum numbers]
{Fractional quantum numbers via complex orbifolds}

	\author{Varghese Mathai}
	\address{School of Mathematical Sciences, University of Adelaide, Adelaide SA 5005 Australia. }
	\email{mathai.varghese@adelaide.edu.au}
	
\author{Graeme Wilkin}
\address{Department of Mathematics, National University of Singapore, Singapore 119076}
\email{graeme@nus.edu.sg}

\begin{abstract}
This paper studies both the conductance and charge transport on 2D orbifolds in a strong magnetic field.  We consider a family of Landau Hamiltonians on a complex, compact 2D orbifold $Y$ that
are parametrised by the Jacobian torus $J(Y)$ of $Y$. We calculate the degree of the associated
stable holomorphic spectral orbibundles when the magnetic field $B$ is large, and obtain
fractional quantum numbers as the conductance and a refined analysis also gives the charge transport.
A key tool studied here is a nontrivial generalisation of the Nahm transform to 2D orbifolds.
\end{abstract}


\maketitle


\thispagestyle{empty}

\baselineskip=16pt


\section*{Introduction}

The fundamental work of Avron, Seiler, Zograf \cite{ASZ} studies a class of  quantum systems for which the transport
coefficients simultaneously display quantization and fluctuation. More precisely,  these models may be
thought of roughly as describing the quantum dynamics of a 
particle on a two dimensional multiply connected compact surface
which has no boundary. They consider the
non-dissipative transport in
a constant magnetic field and  $2h$ Aharonov-Bohm flux
tubes that thread the $h$ handles of the surface. The two related notions of transport coefficients
that they consider are {\em conductance} and {\em charge transport}.
The mathematical tool used in their work is a local families index theorem due to Quillen
\cite{Quillen, A-GMV}. It splits the conductances
into two parts. The first is explicit and  universal, that is, it 
is, up to an integral  factor, the canonical symplectic form on the space of
Aharonov-Bohm fluxes and is quantized,
 therefore providing a
connection to the Integer Quantum Hall effect (IQHE) \cite{ TKNN, ASY, Bellissard}.
For a deeper analysis of the IQHE, see \cite{KMMW}, where the
generating functional, the adiabatic curvature and the adiabatic phase for the IQHE 
is studied on a compact Riemann surface.
The second piece in the formula is a complete
derivative, hence it does not affect charge transport. It affects however
the conductance as a fluctuation term. In contrast to the first, it
depends on spectral properties of the Hamiltonian, as it is related
to the zeta function regularization of its determinant. In \cite{ASZ}, they study
transport coefficients associated to the ground state for any compact Riemann surface
with $h$ handles and where the magnetic field $B$ is a large constant. In \cite{Prieto06}
the general case is considered, again on a compact Riemann surface. In more recent work,  Klevtsov, S. and Wiegmann, P \cite{KW, KW2}
argue that in addition to the Hall conductance and the nondissipative component of the viscous tensor, there exists a third independent transport coefficient, which is precisely quantized and takes constant values along QHE plateaus. They show that it is the Chern number of a vector bundle over the moduli space of surfaces of genus 2 or higher and therefore cannot change continuously along the QHE plateaus.

In this paper, we both use and generalise the results in \cite{ASZ, Prieto06}.  We study a family of Landau Hamiltonians associated to a holomorphic orbiline bundle $L/G$ 
on a complex orbifold $X/G$, where $X$ is a compact Riemann surface and 
$G$ a finite group acting holomorphically and isometrically on $X$ and $L$ is a holomorphic line bundle over $X$
with an induced holomorphic action of $G$. This family is parametrised by the Jacobian variety of $X/G$ which is 
$J(X)^G$, which is the invariant part of the Jacobian variety of $X$. We next generalise the Nahm transform to 2D orbifolds which 
requires a hypothesis on the action of the finite group $G$; see Remark \ref{rem:no-cyclic-quotients}.
It turns out that the spectral bundles $(\hat P)^G_q \to J(X)^G$ associated 
to holomorphic Landau levels of the Landau Hamiltonians are stable holomorphic vector bundles over  $J(X)^G$.  In fact, the spectral bundle
$(\hat P)^G_q$ is the holomorphic Nahm transform of $K_X^{-q} \otimes L$ for all $q\ge 0$ such that $\text{deg}(K_X^{-q} \otimes L) > \text{deg}(K_X)$.

It is argued  that the degree
of the spectral bundles is equal to each other for any $q$ within a specific range. The conductance is associated to a Fermi level is argued using the quantum adiabatic theorem to be equal to 
an integer multiple of the degree of $(\hat P)^G_q$, which is a fraction described  in general in this paper as being $\frac{1}{|G|}\,c_1(\hat P_q)$, where
the degree of the spectral bundle $\hat P_q$ had been calculated previously by Prieto \cite{Prieto06}. The charge transport in this context is also studied in the paper. We mention that the papers where this setting appears experimentally is the recent work of N. Schine et al in Nature 2016 \cite{Schine16}, and in Nature 2019 \cite{Schine19}, where they construct lowest Landau level/QHE states on conic surfaces with rational angle.

We mention alternate approaches to (fractional) quantum numbers on hyperbolic space. For smooth surfaces \cite{CHMM, CHM99, CHM06}, orbifolds \cite{MM99, MM01, MM06}
bulk-boundary correspondence \cite{MT18}, orbifold symmetric products \cite{MS17}. These papers use operator algebras and noncommutative geometry methods, 
in contrast to the holomorphic geometry methods used in this paper.\\

\noindent{\bf Acknowledgements.} VM thanks the Australian Research Council  for support
via the Australian Laureate Fellowship FL170100020. GW was supported by grant number R-146-000-264-114 from the Singapore Ministry of Education Academic
Research Fund Tier 1. This research was also supported by the Australian Research Council Discovery Project DP170101054. GW would also like to thank the University of Adelaide for their hospitality during the development of this paper.

\section{Background on orbifolds}

\subsection{Orbifolds and orbifold bundles}

An \emph{orbifold Riemann surface} is a compact Riemann surface $Y$ of genus $g_Y$, with marked points $\{ p_1, \ldots, p_n \}$ and an isotropy $m_k \in \mathbb{Z}_{>1}$ associated to each $p_k$. For each marked point $p_k$ there is an \emph{orbifold chart} locally modelled on $\mathbb{C} / \mathbb{Z}_{m_k}$, where $Z_{m_k}$ acts on $\mathbb{C}$ by multiplication by $e^{2\pi i / m_k}$. We use $Y^{orb}$ to denote the smooth surface $Y$ together with the marked points and isotropy data $\{ (p_1, m_1), \ldots, (p_n, m_n) \}$. The \emph{orbifold Euler characteristic} of $Y^{orb}$ is
\begin{equation*}
\chi^{orb}(Y^{orb}) = 2-2g-n + \sum_{k=1}^n \frac{1}{m_k}
\end{equation*}

McOwen \cite{McOwen88} proved that this construction has an interpretation in the context of Riemannian geometry: given marked points and isotropies as above, there exists a metric of constant curvature $K \equiv -1$ such that the metric has cone singularities at the marked points with cone angle $\frac{2\pi}{m_k}$. 

If the orbifold Euler characteristic is negative, then the Riemann existence theorem (see for example \cite{Donaldson11}) guarantees the existence of a compact Riemann surface $X$ and a holomorphic map $\pi : X \rightarrow Y$ such that the marked points are branch points and that $\pi$ has order $m_k$ at each branch point $p_k$. If this ramified covering is Galois, then the finite group $G$ of deck transformations acts holomorphically on $X$, and we can identify $Y^{orb} \cong X/G$ where the orbifold charts at the marked points are induced from the action of $G$ on the charts of $X$. The Riemann surface $X$ also depends on a choice of monodromy data around each marked point, and different choices will produce different pairs $(X, G)$ with the same quotient $Y^{orb} = X/G$. The genus of $X$ is fixed, and is related to the orbifold Euler characteristic of $Y^{orb}$ by
\begin{equation*}
\chi(X) = |G| \cdot \chi^{orb}(Y^{orb}) .
\end{equation*}

A complex \emph{orbifold vector bundle} $E \rightarrow Y^{orb}$ has local trivialisations around each marked point $p_k$ with charts
\begin{equation*}
(U \times \mathbb{C}^r) / \mathbb{Z}_{m_k}
\end{equation*}
where $\mathbb{Z}_{m_k}$ acts on $U$ via the orbifold chart described above, and it acts on $\mathbb{C}^r$ via a choice of monodromy representation $\tau : \mathbb{Z}_{m_k} \rightarrow \mathsf{GL}(r, \mathbb{C})$. Since representations of finite groups can be completely reduced (cf. \cite[Cor. 1.6]{FultonHarris91}), then $\tau$ is equivalent to a choice of integers $(d_{k,1}, \ldots, d_{k,r})$ such that $1 \leq d_{k,j} \leq m_k$ for each $j = 1, \ldots, r$.

\subsection{Equivalence between orbifold bundles and equivariant bundles}

The group $G$ acts by holomorphic diffeomorphisms on $X$, and so it acts by pullback on the Jacobian $J(X)$. Furuta and Steer \cite[Thm. 1.3]{FurutaSteer92} show that orbifold bundles over $X/G$ are in one-to-one correspondence with line bundles over $X$ preserved by the $G$-action. Continuing from this, Nasatyr and Steer \cite{NasatyrSteer95} develop a theory of fractional divisors and prove a correspondence between equivalence classes of these divisors and orbifold line bundles with given monodromy representations around the marked points. To an orbifold line bundle with monodromy $e^{2\pi i d_k}{m_k}$ around each marked point $p_k$ is associated a divisor
\begin{equation*}
D = D_{int} + \sum_{k=1}^n \frac{d_k}{m_k} \cdot p_k ,
\end{equation*}
where $D_{int}$ is an integral divisor on the underlying smooth surface $Y$. Conversely, such a divisor determines an orbifold line bundle $L^{orb} \rightarrow Y^{orb}$ up to equivalence. The \emph{degree} of this line bundle is defined to be the degree of the divisor
\begin{equation*}
\deg D = \deg D_{int} + \sum_{k=1}^n \frac{d_k}{m_k}  .
\end{equation*}

For each marked point $p_k$, let $\{ q_1, \ldots, q_{\ell_k} \} = \pi^{-1}(p_k)$ and define the orbifold line bundle $L_{p_k} \rightarrow Y^{orb}$ to be the quotient of $\mathcal{O}_X[q_1 + \cdots q_{\ell_k}]$ by the action of $G$. Note that $|G| = \ell_k m_k$, and therefore $\deg L_{p_k} = \frac{1}{m_k}$. This leads to a correspondence between holomorphic line bundles $L \rightarrow Y$ and orbifold line bundles  $L^{orb} \rightarrow Y^{orb}$ with given monodromy $(d_1, \ldots, d_n)$. The correspondence is given explicitly by
\begin{equation}\label{eqn:orbifold-holomorphic-correspondence}
L = L^{orb} \otimes L_{p_1}^{-d_1} \otimes \cdots \otimes L_{p_n}^{-d_n} ,
\end{equation}
which determines an isomorphism of sheaves $\mathcal{O}_Y(L) \cong \mathcal{O}_{Y^{orb}}(L^{orb})$ (cf. \cite[Prop. 1.5]{NasatyrSteer95}). 

Therefore we have a correspondence between line bundles on $Y$ of degree $d$ and orbifold line bundles on $Y^{orb} = X/G$ with degree $d^{orb} := d + \sum_{k=1}^n \frac{d_k}{m_k}$, which we write as $\Pic_d(Y) \stackrel{\cong}{\rightarrow} \Pic_{d^{orb}}(Y^{orb})$. 

The group $G$ acts on the Jacobian $J(X)$ by pullback, and \cite[Thm. 1.3]{FurutaSteer92} shows that $\Pic_{d^{orb}}(Y^{orb}) \cong J(X)^G$ . If we suppose that the ramified cover $f : X \rightarrow Y$ does not factorise through an e\'tale cover, then the pullback $f^* : J(Y) \rightarrow J(X)$ is injective (cf. \cite[Prop. 11.4.3]{BirkenhakeLange04}), and so the following map defines an isomorphism of abelian varieties
\begin{align*}
\Pic_d(Y) & \longrightarrow \Pic_{d^{orb}|G|}(X)^G \\
L & \mapsto f^* L \otimes \bigotimes_{k=1}^n \mathcal{O}_X[d_k(q_1 + \cdots + q_{\ell_k})] .
\end{align*}
Fixing divisors of degree $d$ on $Y$ and degree $d^{orb} |G|$ on $X$ gives an isomorphism $J(Y) \cong J(X)^G$.

\subsection{Orbifold Riemann-Roch and isotypic components}

The \emph{canonical bundle} of $Y^{orb}$ is 
\begin{equation*}
K_{Y^{orb}} := K_Y[p_1 + \cdots + p_n] \otimes L_{p_1}^{-1} \otimes \cdots \otimes L_{p_n}^{-1} .
\end{equation*}
The degree of the canonical bundle is then
\begin{equation*}
\deg K_{Y^{orb}} = 2g-2 + n - \sum_{k=1}^n \frac{1}{m_k} = - \chi^{orb}(Y^{orb}) = \frac{1}{|G|} \cdot \deg K_X .
\end{equation*}



Using the correspondence of \cite[Thm. 1.3]{FurutaSteer92}, let $\tilde{L} \rightarrow X$ denote the $G$-equivariant bundle corresponding to $L^{orb} \rightarrow X/G$, and let $L \rightarrow Y$ denote the smooth bundle over the Riemann surface $Y$ corresponding to $X/G$ via the construction of \eqref{eqn:orbifold-holomorphic-correspondence}.
The \emph{orbifold Riemann-Roch} theorem is 
\begin{align}\label{eqn:orbifold-riemann-roch}
\begin{split}
\dim_\C H^0(L^{orb}) - \dim_\C H^1(L^{orb}) & = 1-g + \deg^{orb}(L^{orb}) - \sum_{j=1}^n \frac{d_j}{m_j} \\
 & = 1 - g + \deg(L) .
\end{split}
\end{align}

In particular, the correspondence of sheaves from \cite[Prop. 1.5]{NasatyrSteer95} shows that 
\begin{equation*}
H^j(\tilde{L})^G =: H^j(L^{orb}) \cong H^j(L) .
\end{equation*}

When the degree of $\tilde{L}$ satisfies $\deg \tilde{L} > -\chi(X)$, or equivalently $\deg(L^{orb}) > \chi(Y^{orb})$, then $h^1(L^{orb}) = 0$, in which case the orbifold Riemann-Roch formula \eqref{eqn:orbifold-riemann-roch} computes the dimension of $H^0(\tilde{L})^G = H^0(L^{orb})$. One can also take this a step further and use the Atiyah-Bott fixed point formula \cite{AtiyahBott67}, \cite{AtiyahBott68} to compute the dimensions of the isotypic components.

\subsection{Holomorphic Landau levels}

In this section we show that the holomorphic Landau levels, which correspond to eigensections of the Laplacian on an equivariant bundle, are isomorphic to spaces of holomorphic sections of an associated bundle, in analogy with the non-equivariant case studied by Prieto \cite{Prieto06}, \cite{Prieto06-2}.

From the non-equivariant case (cf. \cite[Sec. 2]{Prieto06}), on the surface $X$ with scalar curvature $R$, there is a Schr\"odinger operator
\begin{equation}\label{eqn:upstairs-schrodinger-operator}
\hat{H} = \frac{\hbar^2}{2m} \left( \nabla^* \nabla + \frac{R}{6} \right)
\end{equation}
defined on the line bundle $\tilde{L} \rightarrow X$. By varying the flat connection, we obtain a Hilbert bundle of sections $\hat{\mathcal{H}} \rightarrow \Pic^0(X)$ with fibre $L^2(X, \tilde{L})$. Prieto \cite[Thm. 7]{Prieto06-2} shows that if $q > 0$ and $\deg (\tilde{L} \otimes K_X^{-q}) > 2g_X-2$, then the eigensections of $\hat{H}$ with eigenvalue $E_q$ given by \cite[Thm. 7]{Prieto06-2} are precisely the holomorphic sections of $\tilde{L} \otimes K_X^{-q}$. Therefore the eigenvalues in this range are discrete and their multiplicity can be computed by Riemann-Roch. The key to this calculation is the following identity for smooth sections $s \in \Omega^0(\tilde{L})$ in terms of the operators $\Delta_{\bar{\partial}} := \bar{\partial}^* \bar{\partial}$ and $\Delta_{\partial} := \partial^* \partial$, and the curvature $F_\nabla$ of the connection $\nabla = \bar{\partial} + \partial$ on $\tilde{L}$
\begin{equation*}
\Delta_{\partial} s - \Delta_{\bar{\partial}} s = i *F_{\nabla} s .
\end{equation*}
Then $\nabla^* \nabla s = 2 \Delta_{\bar{\partial}} s + i * F_{\nabla} s$, and so if the curvature is constant we see that $i*F_\nabla$ is a lower bound for the eigenvalues of $\nabla^* \nabla$, and that if $i*F_\nabla > 0$ then this eigenvalue is achieved if and only if $s$ is holomorphic. The eigenvalue can be computed using the Chern-Weil theory, which shows that $i *F_\nabla = \frac{2\pi \deg(\tilde{L})}{\vol(X)}$. 

If $\deg \tilde{L} > 2g-2$, then the next eigenvalue can also be expressed in this way. Write $\nabla^{(-1)}$ for the induced connection on $\tilde{L} \otimes K_X^{-1}$, $\partial^{\nabla^{(-1)}}$ for the $(1,0)$ part of $\nabla^{(-1)}$ and $\Delta_{\bar{\partial}}^{(-1)}$ for the associated $\bar{\partial}$-Laplacian. On the orthogonal complement of $\ker \bar{\partial}$, we can write $s = \partial^{\nabla^{(-1)}} s'$ for some $s' \in \tilde{L} \otimes K_X^{-1}$. The identity
\begin{equation*}
\Delta_{\bar{\partial}} \circ \partial^{\nabla^{(-1)}} = \partial^{\nabla^{(-1)}} \circ \Delta_{\bar{\partial}}^{(-1)} + i \partial^{\nabla^{(-1)}} * F_{\nabla^{(-1)}}
\end{equation*}
then shows that $s = \partial^{\nabla^{(-1)}} s'$ is an eigensection of $\nabla^* \nabla$ with eigenvalue $2 i*F_{\nabla^{(-1)}} + i*F_{\nabla}$ if and only if $s' \in \Omega^0(\tilde{L} \otimes K_X^{-1})$ is holomorphic. The dimension of the space of eigensections can then be computed using Riemann-Roch. Repeating this process shows that if $\deg (\tilde{L} \otimes K_X^{-q}) > 0$ for some $q \in \mathbb{Z}_{>0}$, then the low-lying eigenvalues of $\nabla^* \nabla$ are given by
\begin{equation}\label{eqn:low-eigenvalue}
E_\ell :=  \frac{2\pi}{\vol(X)}\left( (2 \ell + 1) \deg \tilde{L} - \ell(\ell+1) (2g-2) \right) = i*F_\nabla + 2 \sum_{k=1}^\ell i*F_{\nabla^{-k}} .
\end{equation}
for any $0 \leq \ell \leq q$, and the multiplicity of each eigenvalue can be computed using Riemann-Roch.

On restricting to the $G$-equivariant setting, the Schr\"odinger operator \eqref{eqn:upstairs-schrodinger-operator} on a $G$-invariant bundle $\tilde{L} \rightarrow X$ descends to the orbifold $Y^{orb} = X/G$, and the $G$-invariant eigensections become $G$-invariant sections of the orbifold bundle $L^{orb} \rightarrow Y^{orb}$. The $G$-invariant holomorphic sections of $\tilde{L} \otimes K_X^{-q}$ correspond to the $G$-invariant eigensections of the Schr\"odinger operator with the same eigenvalue $E_q$ given by \eqref{eqn:low-eigenvalue}, and the multiplicity can be computed using the orbifold Riemann-Roch theorem \eqref{eqn:orbifold-riemann-roch}. Let $L_q \rightarrow Y$ be the smooth bundle associated to $L^{orb} \otimes K_{Y^{orb}}^{-q}$ by the correspondence of \eqref{eqn:orbifold-holomorphic-correspondence}. Note that if $L \rightarrow Y$ is the smooth bundle associated to $L^{orb}$, then
\begin{equation*}
\deg (L^{orb} \otimes K_{Y^{orb}}^{-q}) = \deg L + \sum_{k=1}^n \frac{d_k}{m_k} - q(2g-2) - qn + \sum_{k=1}^n \frac{q}{m_k} ,
\end{equation*}
and so 
\begin{equation*}
\deg L_q = \deg L - q(2g-2) - qn + \sum_{k=1}^n \left\lfloor \frac{d_k+q}{m_k} \right\rfloor ,
\end{equation*}
where the floor function $\lfloor x \rfloor$ denotes the largest integer less than or equal to $x$. The orbifold Riemann-Roch theorem then shows that if $\deg (\tilde{L} \otimes K_X^{-(q+1)}) > 0$, then the multiplicity $M_q$ of the eigenvalue $E_q$ is
\begin{equation*}
M_q = 1-g + \deg L_q. 
\end{equation*}

We then have the following orbifold version of \cite[Thm. 5]{Prieto06}. The same proof as \cite[Lem. 1]{Prieto06} then shows that we can apply the adiabatic theorem in this setting.

\begin{lemma}\label{lem:orbifold-eigensections}
Let $L^{orb} \rightarrow Y^{orb} = X/G$ be an orbifold line bundle. If $\deg^{orb}(L^{orb}) > -\chi^{orb}(Y^{orb}) = \frac{1}{|G|} (2g_X-2)$, then the orbifold eigensections of the Schr\"odinger operator \eqref{eqn:upstairs-schrodinger-operator} with eigenvalue $E_q$ given by \eqref{eqn:low-eigenvalue} correspond to orbifold holomorphic sections of $L^{orb} \otimes K_{Y^{orb}}^{-q}$.
\end{lemma}

\begin{corollary}\label{cor:orbifold-adiabatic}
In the orbifold setting, the conditions (A1)--(A3) from \cite[Sec. 3]{Prieto06} are satisfied. In particular, the low-lying eigenvalues are separated by a spectral gap and the projection operator to the eigenspaces is of finite rank.
\end{corollary}

\section{Fourier-Mukai transform of an orbifold line bundle}

Using the notation and setup of the previous section, let $G$ be a finite group acting on $X$ by holomorphic diffeomorphisms, and let $\pi : X \rightarrow X / G$ denote the quotient map. In addition, we assume that if $G' \subset G$ is a normal subgroup such that $Q := G / G'$ is cyclic, then $Q$ does not act freely on $X / G'$. The group $G$ then acts on the Jacobian $J(X)$ by pullback. We also fix monodromy data around the points with nontrivial isotropy group to fix the correspondence \eqref{eqn:orbifold-holomorphic-correspondence}.

Let $Y$ be the smooth Riemann surface underlying the quotient $X / G$. The above assumption on the action of a cyclic quotient group $Q$ implies that the ramified cover $f : X \rightarrow Y$ does not factorise through an \'etale cover (cf. \cite[Prop. 11.4.3]{BirkenhakeLange04}), therefore the pullback $f^* : J(Y) \rightarrow J(X)^G$ is an isomorphism of abelian varieties. 
\begin{equation*}
\begin{tikzcd}[row sep=small]
 & X \arrow{dl}[swap]{\pi} \arrow{dr}{f} & \\
X/G \arrow{rr} & & Y
\end{tikzcd}
\end{equation*}

The Poincare bundle $\mathcal{P}_X \rightarrow X \times J(X)$ is uniquely characterised by the property that the restriction to $X \times \{L\}$ is the bundle $L \rightarrow X$. The pullback $i^* \mathcal{P}_X$ by the inclusion $i : J(X)^G \hookrightarrow J(X)$ satisfies the same property for $G$-invariant bundles. Therefore $i^* \mathcal{P}_X$ is the pullback of the orbifold Poincar\'e bundle $\mathcal{P}_{X/G} \rightarrow X/G \times J(X)^G$ by the projection $\pi \times \id : X \times J(X)^G \rightarrow X/G \times J(X)^G$. 

Using the identification $J(Y) \cong J(X)^G$, we see that $i^* \mathcal{P}_X = (f \times \id)^* \mathcal{P}_Y$. Therefore $\mathcal{P}_{X/G}$ is the pullback of $\mathcal{P}_Y$ by the cartesian product of the smoothing map $X/G \rightarrow Y$ with the isomorphism $J(X)^G \cong J(Y)$. In particular, this defines an isomorphism of Poincar\'e bundles $\mathcal{P}_{X/G} \cong \mathcal{P}_Y$. 
\begin{equation*}
{\footnotesize
\begin{tikzcd}[row sep=small]
 & i^* \mathcal{P}_X \arrow{dl} \arrow{dr} \arrow{dd} & \\
\mathcal{P}_{X/G} \arrow{dd} \arrow[dashrightarrow,crossing over]{rr} & & \mathcal{P}_Y \arrow{dd} \\
 & X \times J(X)^G \arrow{dl}[swap]{\pi \times \id} \arrow{dr}{f \times \id} & \\
X/G \times J(X)^G \arrow{rr} & & Y \times J(Y)
\end{tikzcd}}
\end{equation*}

Now we compare the different Fourier-Mukai transforms associated to these Poincar\'e bundles. First consider a $G$-invariant line bundle $\tilde{L} \rightarrow X$, with $\deg \tilde{L} > \chi(X)$.
\begin{equation*}
\begin{tikzcd}[row sep=small]
 & X \times J(X)^G \arrow{dl}[swap]{p_1} \arrow{dr}{p_2} & \\
X & & J(X)^G
\end{tikzcd}
\end{equation*}

\begin{definition}
The \emph{$G$-equivariant Fourier-Mukai transform} of $\tilde{L}$ is 
\begin{equation*}
\hat{\tilde{L}}^G := (p_2)_* (p_1^* \tilde{L} \otimes i^* \mathcal{P}_X) .
\end{equation*}
\end{definition}

This is a bundle over $J(X)^G$, where the fibre over $L' \in J(X)^G$ is the space of $G$-invariant sections $H^0(\tilde{L} \otimes L')^G$. 

Using the correspondence \eqref{eqn:orbifold-holomorphic-correspondence} from the previous section, the $G$-invariant bundle $\tilde{L} \rightarrow X$ corresponds to an orbifold line bundle $L^{orb} \rightarrow X/G$, which is isomorphic as a sheaf to a holomorphic bundle $L \rightarrow Y$, satisfying $\deg L > -\chi(Y)$. The ordinary Fourier-Mukai transform of $L$ is the spectral bundle $\hat{L} \rightarrow J(Y) \cong J(X)^G$, where the fibre over $L'' \in J(Y)$ is the space of sections $H^0(L \otimes L'')$. Let $L' \in J(X)^G$ be the bundle corresponding to $L''$ under the isomorphism $J(Y) \cong J(X)^G$. Then the correspondence \eqref{eqn:orbifold-holomorphic-correspondence} shows that sections of $L \otimes L''$ correspond to orbifold sections of $L \otimes L'$, which are $G$-invariant sections of $\tilde{L} \otimes \pi^* L'$. Therefore the two spectral bundles $\hat{L} \rightarrow J(Y)$ and $\hat{L}^G \rightarrow J(X)^G$ have isomorphic fibres. 

Now we extend this to an isomorphism of bundles. Given an open set $U \subset J(Y) \cong J(X)^G$, the sections of $\mathcal{O}_{\hat{L}}(U)$ are by definition the sections of $p_1^* L \otimes \mathcal{P}_Y$ over $p_2^{-1}(U)$. These are orbifold sections of $p_1^* L^{orb} \otimes \mathcal{P}_{X/G}$, or equivalently $G$-invariant sections of $p_1^* \tilde{L} \otimes i^* \mathcal{P}_X$ over $p_2^{-1}(U)$. Therefore the spectral bundles $\hat{L} := (p_2)_* (L \otimes \mathcal{P}_Y) \rightarrow J(Y)$ and $\hat{L}^G := (p_2)_* (L^{orb} \otimes \mathcal{P}_{X/G}) \rightarrow J(X)^G$ are isomorphic, since their sheaves of sections are isomorphic.

In particular, their degrees and ranks are the same. The calculation from \cite[Prop. 4]{Prieto06} then shows that
\begin{equation*}
c_1(\hat{L}) = - [\Theta_Y] \in H^2(J(Y), \mathbb{Z}) ,
\end{equation*}
and Lemma 12.3.1 of \cite{BirkenhakeLange04} shows that $[\Theta_Y] = \frac{1}{|G|} i^*[\Theta_X]$. Therefore we have proved
\begin{proposition}\label{prop:equivariant-spectral-c1}
Let $X$ be a compact Riemann surface and $G$ a finite group acting on $X$ by holomorphic diffeomorphisms, such that if $G' \subset G$ is a normal subgroup such that $Q := G / G'$ is cyclic, then $Q$ does not act freely on $X / G'$. Let $L^{orb} \rightarrow X / G$ be an orbifold line bundle of degree greater than $-\chi^{orb}(X/G)$. Then the rank of the spectral bundle $\hat{L}^G \rightarrow J(X)^G$ is given by \eqref{eqn:orbifold-riemann-roch}
\begin{equation*}
\rank(\hat{L}^G) = 1-g+\deg^{orb}(L^{orb}) - \sum_{j=1}^n \frac{d_j}{m_j} = 1-g + \deg L
\end{equation*}
and the first Chern class is
\begin{equation*}
c_1(\hat{L}^G) = -\frac{1}{|G|} i^*[\Theta_X] \in H^2(J(X)^G, \mathbb{Z}).
\end{equation*}
\end{proposition}

\begin{remark}\label{rem:no-cyclic-quotients}
The condition that cyclic quotient groups $Q = G / G'$ do not act freely guarantees that the ramified covering $f : X \rightarrow Y$ does not factor through a cyclic \'etale cover $f' : X' \rightarrow Y$, and hence $f^* : J(Y) \rightarrow J(X)^G$ is an isomorphism. The existence of such an \'etale cover implies that $f^*$ is not injective (cf. \cite[Prop. 11.4.3]{BirkenhakeLange04}). 

A simple example to illustrate this is given as follows. Let $Y = \C / (\mathbb{Z} + \tau \mathbb{Z})$ and $X = \C / (n \mathbb{Z} + \tau \mathbb{Z})$ be two elliptic curves, with $n$-fold covering $p : X \rightarrow Y$ induced from the identity map on $\C$. Let $D_\ell := \sum_k m_k y_k \in \Div(Y)$ with $\deg D_\ell = \sum_k m_k = 0$ be a divisor in the equivalence class $\sum_k m_k y_k = a + \frac{\ell}{n} \tau \in J(Y) \cong Y$. The pullback is $p^* D_\ell = \sum_k m_k (y_k^{(1)} + \cdots + y_k^{(n)}) \in \Div(X)$ with $y_k^{(j)} = y_k^{(1)} + j-1$ for each $j = 1, \ldots, n$. Then the equivalence class in the Jacobian is
\begin{equation*}
p^* D_\ell = n \sum_k m_k y_k + (1 + \cdots + n-1) \sum_k m_k = n \sum_k m_k y_k = na + \ell \tau \sim na \in J(X) ,
\end{equation*}
which is independent of the choice of $\ell = 0, \ldots, n-1$, and therefore $p^* : J(Y) \rightarrow J(X)^G$ is not injective.
\end{remark}

\section{Mean transport and conductivity on an orbifold}

Prieto \cite[Thm. 7]{Prieto06} gives a geometric formula for the mean adiabatic transport $\left< Q_{Ad} \right>$ and mean Hall conductivity $\left< \sigma_H \right>$ in terms of the geometry of the Jacobian of the underlying surface. In order to apply Kato's adiabatic theorem, the assumptions (A1)--(A3) from \cite[Sec. 3]{Prieto06} are needed, in particular the existence of spectral gaps and finite rank eigenspaces, which Corollary \ref{cor:orbifold-adiabatic} shows are satisfied in the orbifold setting. Then the adiabatic transport is asymptotic to the mean adiabatic transport $\left< Q_{Ad}(\hat{L})^G \right>$ (cf. \cite[Thm. 2]{Prieto06}), which we now compute, following 
\cite{Prieto06} and our earlier constructions.

\begin{theorem}
Let $X$ be a compact Riemann surface and $G$ a finite group acting on $X$ satisfying the same conditions as Proposition \ref{prop:equivariant-spectral-c1}. Let $p : X \rightarrow Y$ denote the associated ramified covering map. For $\beta, \delta \in H_1(Y) \cong H_1(J(Y)) \cong H_1(J(X)^G) \stackrel{i_*}{\hookrightarrow} H_1(J(X)) \cong H_1(X)$ we have
\begin{align}
\left< Q_{Ad}(\hat{L})^G \right> (\beta, \delta) & = -\frac{e}{|G|} (i_* \beta, i_* \delta)_X \label{eqn:orbifold-charge-transport} \\
\left< \sigma_H(\hat{L})^G \right> (\beta, \delta) & = - \frac{1}{|G|} \frac{e^2}{h}(i_* \beta, i_* \delta)_X . \label{eqn:orbifold-mean-conductivity}
\end{align}

\end{theorem}

\begin{proof}
A symplectic basis $\{ \alpha_1, \ldots, \alpha_{2g} \}$ for $H_1(Y)$ determines a dual basis for $H^1(Y, \mathbb{Z})$ which determines a basis $\{ \gamma_1, \ldots, \gamma_{2g} \}$ for $H_1(J(Y)) \cong H_1(J(X)^G) \hookrightarrow H_1(J(X))$. The Pontryagin product is denoted $\gamma_k \star \gamma_j \in H_2(J(X)^G)$. On $J(X)^G$, the mean adiabatic transport is written in terms of the curvature of the spectral bundle
\begin{equation*}
\left< Q_{Ad}(\hat{L}_q^G) \right>(\alpha_k, \alpha_j) = e \frac{i}{2\pi} \int_{\gamma_k \star \gamma_j} \Omega^{\hat{L}_q^G} .
\end{equation*}
This integral can be written in terms of the first Chern class from Proposition \ref{prop:equivariant-spectral-c1}.
\begin{equation*}
e \frac{i}{2\pi} \int_{\gamma_k \star \gamma_j} \Omega^{\hat{L}_q^G} = e \int_{\gamma_k \star \gamma_j} c_1(\hat{L}_q^G) = -\frac{e}{|G|} \int_{\gamma_k \star \gamma_j} i^* [\Theta_X] = -\frac{e}{|G|} \int_{i_*(\gamma_k) \star i_*(\gamma_j)} [\Theta_X].
\end{equation*}
Therefore, for any $\beta, \delta \in H_1(Y) \cong H_1(J(Y)) \cong H_1(J(X)^G)$ we have
\begin{equation*}
\left< Q_{Ad}(\hat{L}_q^G) \right>(\beta, \delta) = \frac{e}{|G|} (i_* \beta, i_* \delta)_X ,
\end{equation*}
where $( \cdot, \cdot )_X$ is the intersection form on $H_1(X)$, isomorphic to the pairing on $H_1(J(X)) \cong H_1(X)$ determined by the theta divisor. The second identity is proved similarly (cf. \cite[Rem. 5 \& Thm. 7]{Prieto06}).
\end{proof}

\section*{Conclusions and outlook}

Both the papers by Avron, Seiler, Zograf \cite{ASZ} and Prieto \cite{Prieto06} assume that the magnetic field $B$ is an integral 2-form on 
the compact Riemann surface $X$, such that $\int_X B>>0$ is large positive. We relax the integrality assumption in this paper where
$\int_X B \in \ZZ/|G|$ is rational but also assume that $\int_X B>>0$ is large positive.  Ideally, we would like to remove the rationality constraint on 
 $\int_X B $ and we are discussing future work related to this. We also would like to extend our conclusions to finite volume Riemann surfaces.
 
Another goal is to generalise the Hamiltonians considered here, by for instance adding an electic potential term. This would allow for crossings in the 
 spectral orbibundles, whose analysis could be related to the study of semimetals (cf. \cite{MT171, MT172} in the Euclidean setting, and the fundamental article
 \cite{TV} and review article \cite{AMV}) in the holomorphic setting.


\end{document}